\documentclass[12pt]{amsart}

\usepackage{amssymb}
\usepackage{amsmath}
\usepackage{amsthm}
\usepackage{amsbsy}
\usepackage{bm}
\usepackage{hyperref}
\date{\today}
\usepackage{cite}
\usepackage{array}
\usepackage{fullpage}
\date{\today}
\newcommand{\comment}[1]{}

\theoremstyle{theorem}
    \newtheorem{theorem}{Theorem}
    \newtheorem{lemma}[theorem]{Lemma}
    
    \newtheorem{corollary}[theorem]{Corollary}

\theoremstyle{definition} 
    \newtheorem{definition}[theorem]{Definition}
    \newtheorem{fact}[theorem]{Fact}
    
    \newtheorem{remark}[theorem]{Remark}
    \newtheorem{example}[theorem]{Example}
    \newtheorem{exercise}[theorem]{Exercise}

\def\R{\mathbb{R}}

\def\b{{\bf b}}

\def\a{{\sigma}}
\def\b{{\lambda}}

\def\l{\left}
\def\r{\right}
\def\<{\langle}
\def\>{\rangle}

\newcommand{\E}{\mbox{\bf E}}

\def\bar{\overline}
\def\P{{\bf P}}

\newcommand\Tr{{\mbox{Tr}}}

\newcommand\mnote[1]{} 
\newcommand\be{\begin{equation*}}

\newcommand\ee{\end{equation*}}

\newcommand\ben{\begin{equation}}
\newcommand\een{\end{equation}}
\newcommand\bes{\begin{eqnarray*}}
\newcommand\ees{\end{eqnarray*}}

\newcommand\bex{\begin{exercise}}
\newcommand\eex{\end{exercise}}
\newcommand\beg{\begin{example}}
\newcommand\eeg{\end{example}}
\newcommand\benu{\begin{enumerate}}
\newcommand\eenu{\end{enumerate}}
\newcommand\beit{\begin{itemize}}
\newcommand\eeit{\end{itemize}}
\newcommand\berk{\begin{remark}}
\newcommand\eerk{\end{remark}}
\newcommand\bdefn{\begin{defintion}}
\newcommand\edefn{\end{definition}}
\newcommand\bthm{\begin{theorem}}
\newcommand\ethm{\end{theorem}}
\newcommand\bprf{\begin{proof}}
\newcommand\eprf{\end{proof}}
\newcommand\blem{\begin{lemma}}
\newcommand\elem{\end{lemma}}

\newcommand{\sm}{{\raise0.3ex\hbox{$\scriptstyle \setminus$}}}

\def\l{\left}
\def\r{\right}

\def\CHI{\mathchoice%
{\raise2pt\hbox{$\chi$}}%
{\raise2pt\hbox{$\chi$}}%
{\raise1.3pt\hbox{$\scriptstyle\chi$}}%
{\raise0.8pt\hbox{$\scriptscriptstyle\chi$}}}
\def\smalloplus{\raise1pt\hbox{$\,\scriptstyle \oplus\;$}}
\allowdisplaybreaks
\title{Lyapunov exponents and eigenvalues of products of random  matrices}

\author{Nanda Kishore Reddy}

\address{Department of Mathematics\\
        Indian Institute of Science\\
        Bangalore 560012, India}

\email{kishore11@math.iisc.ernet.in}

\thanks{Partially supported by UGC Centre for Advanced Studies. Research of Nanda Kishore reddy is supported by CSIR-SPM fellowship, CSIR, Government of India.}

\begin{document}
\maketitle

\begin{abstract}
 Let $X_1,X_2, \ldots $ be a sequence of  $i.i.d$ real (complex) $d \times d $  invertible random matrices with common distribution $\mu$ and  $\a_1(n), \a_2(n), \ldots , \a_d(n)$ be  the singular values, $\b_1(n), \b_2(n), \ldots , \b_d(n)$ be  the eigenvalues of $X_nX_{n-1}\cdots X_1$ in the decreasing order of their absolute values for every $n$. It is known that if $\E(\log^{+}\|X_1\|)< \infty$, then  with probability one for all $1 \leq p \leq d$,
$$
 \lim_{n \to \infty} \frac{1}{n}\log \a_p(n)=\gamma_p,
$$ 
where  ${\gamma_1,\gamma_2 \ldots \gamma_d}$ are the Lyapunov exponents associated with $\mu$. In this paper we show that under certain support and moment conditions on $\mu$, the absolute values of eigenvalues also exhibit the same asymptotic behaviour. In fact, a stronger asymptotic relation holds between the singular values and the eigenvalues $i.e.$ for any $r>0$ with probability one for all $1 \leq p \leq d$,
$$
\lim_{n \to \infty} \frac{1}{n^r}\log \l(\frac{|\b_p(n)|}{\a_p(n)}\r)= 0,
$$
which implies that  the fluctuations of the eigenvalues have the same asymptotic distribution as that of the corresponding singular values. Isotropic random matrices and also random matrices with $i.i.d$ real elements,  which have some finite moment and bounded density whose support contains an open set, are shown to satisfy the moment and support conditions under which the above relations hold. 
\end{abstract}
\section{Introduction}
Lyapunov exponents and singular values of products of $i.i.d$ random matrices of fixed order have been studied extensively (see \cite{prm}, \cite{kesten}). Only recently, the study of eigenvalues of such matrices has begun, though  they have been considered first in the setting of dynamical systems in \cite{goldman}  and there the absolute values of eigenvalues have been conjectured to show the same asymptotic behaviour as that of singular values.  Recent comparative studies \cite{akemann1}, \cite{ipsen} of  singular values and eigenvalues  in the case of Ginibre matrix ensembles  have verified the conjecture to be true  in the respective cases. And also  \cite{forrester2} mentions this in the case of random truncated unitary matrices. For a summary of results on this topic, we refer the reader to \cite{ipsen2}. In \cite{akemann1},  asymptotic distributions of first order fluctuations of (logarithm of) both  the singular values and the  absolute values  of eigenvalues for product of  Ginibre matrices have been computed and shown to be equal and also a proof of this conjecture in the case of $2 \times 2$ isotropic random matrices  has been given. Recently, a full proof in the case of isotropic random matrices has been obtained in \cite{nanda}. In this paper we obtain a proof in more generality with only few conditions on the moments and  support of the distribution of the  random matrices involved.

The paper is organised as follows. In Section \ref{section1} we provide definitions required to describe those moment and support conditions on the distribution  of the random matrices. we also provide many useful facts there (all taken from \cite{prm}). In Section \ref{section2} we prove a lemma (\ref{lemma}) regarding the right and left singular vectors of the product random matrices. Using the lemma, we obtain a proof for the asymptotic relation (\ref{theorem2}) between the singular values and  the absolute values of the eigenvalues of product random matrices. In the case of random matrices with $i.i.d$ elements, we  derive sufficient conditions on the distribution of elements for the result to hold. Throughout the paper we considered only real random matrices. As the facts quoted here from \cite{prm} carry over to the complex case also, Lemma \ref{lemma} and Theorem \ref{theorem2} can easily be verified to be true in the complex case as well.

\section{Definitions and Facts}\label{section1}

Consider a sequence  $\lbrace Y_n, n\geq 1 \rbrace$ of $i.i.d$ $d \times d$ real random matrices with common distribution $\mu$. Let $S_n=Y_n\cdots Y_1$ for all $n\geq 1$. Then the Lyapunov exponents are defined as follows. 
\begin{definition}
If $\E(\log^{+}\|Y_1\|)< \infty$ (we write $a^+$ for $\max (a,0)$), the {\bf first Lyapunov exponent} associated with $\mu$ is the element ${\gamma}$ of $\R\cup \{-\infty\}$ defined by 
$$
\gamma= \lim_{n \to \infty} \frac{1}{n}\E(\log \|S_n\|)
$$
\end{definition}
The norm used is spectral norm. Since all norms on the finite dimensional vector spaces are equivalent, $\gamma$ is independent of the chosen norm.  The existence of the limit in the definition follows from the subadditivity of the  sequence $\{\E(\log \|S_n\|), n\geq 1\}$ and Fekete's subadditive lemma. The definitions of next Lyapunov  exponents involve  exterior powers of matrices. 

Given a matrix $M$, the $p$-th exterior power $\wedge^p M$ ($1 \leq p \leq d$) is defined to be the ${d \choose p} \times {d \choose p}$ matrix, whose rows and columns are indexed by $p$-sized subsets of $\{1,2 \ldots d\}$ in dictionary order, such that $I,J$-th element is given by 
$$
\wedge^p M_{I,J}=[M]_{I,J},
$$
where $[M]_{I,J}$ denotes the $p \times p$ minor of $M$ that corresponds to the rows  with index in $I$ and the columns with index in $J$. ($[M]_I$ denotes $[M]_{I,I}$). By Cauchy-Binet formula, it can be seen that 
$$\wedge^p (MN)= (\wedge^p M)(\wedge^p N).$$
Let $M=U\Sigma V^*$ be a singular value decomposition of $M$ and $\a_1(M)\geq \a_2(M) \cdots \geq \a_d(M)$ be singular values of $M$, then 
$$\wedge^p M=(\wedge^p U)(\wedge^p \Sigma)(\wedge^p V^*)
$$ 
is a singular value decomposition of $\wedge^p M$. Therefore $$\|\wedge^p M\|=\a_1(M) \cdots  \a_p(M).$$ Define ${ {\ell(M)}= \max (\log^+\|M\|, \log^+\|M^{-1}\|)}.$ Observe that $\frac{1}{p}|\log \|\wedge^p M\|| \leq  \ell(M)$ for all $1 \leq p \leq d$.
 Let $\mu$ be as in previous definition.  

\begin{definition}
The {\bf Lyapunov exponents} ${\gamma_1,\gamma_2 \ldots \gamma_d}$ associated with $\mu$ are defined inductively by $\gamma=\gamma_1$ and for $p\geq 2$ 
$$\sum_{i=1}^p \gamma_i=\lim_{n \to \infty} \frac{1}{n}\E(\log \|\wedge^p S_n\|).
$$
If for some $p$ $\sum_{i=1}^{p-1} \gamma_i= -\infty$, we put $\gamma_p=\gamma_{p+1}=\cdots \gamma_d= -\infty$. 
\end{definition}

Note that if $Y_1$ is  invertible with probability one and $\E(\log ^{+}\|Y_1^{-1}\|)$ is also finite, then for any $p$
$$
\frac{1}{p}|\sum_{i=1}^p \gamma_i| \leq  \E(\ell(Y_1)) < \infty
$$
so that all the Lyapunov exponents are finite. 

The following fact gives almost sure convergence in the definition of Lyapunov exponents, if  $Y_1$ is invertible random matrix. (See Theorem 4.1, Chapter I, \cite{prm}).
\begin{fact}
Let $Y_1,Y_2,\ldots$ be $i.i.d$ matrices in $ GL_d(\R)$ and $S_n=Y_n\cdots Y_1$. If $\E(\log^{+}\|Y_1\|)< \infty$, then with probability one 
$$
\lim_{n \to \infty} \frac{1}{n}\log \|S_n\|= \gamma_1.
$$
\end{fact}
 The above fact, when applied to the  exterior powers of the random matrices considered there, gives the following fact (see Proposition 5.6, Chapter III, \cite{prm}).
 
\begin{fact}\label{fact1}
If $\a_1(n)\geq \a_2(n)\geq \cdots \geq \a_d(n) >0$ are singular values of $S_n$ (with the assumptions of  the previous fact), then with probability one, for $1 \leq p \leq d$
$$
\gamma_p= \lim_{n \to \infty} \frac{1}{n}\E(\log \a_p(n))= \lim_{n \to \infty} \frac{1}{n}\log \a_p(n).
$$
\end{fact}

\begin{remark}\label{remark1} The above fact implies $\gamma_1 \geq \gamma_2 \geq \cdots \geq \gamma_d$.
It can also be seen from the above fact that the first two Lyapunov exponents associated with the measure of $\wedge^p Y_1 $ are $\sum_{i=1}^{p-1} \gamma_i + \gamma_p$ and $ \sum_{i=1}^{p-1} \gamma_i + \gamma_{p+1}$ for any $1 \leq p \leq d$.
\end{remark}

One immediate question would be whether the absolute values  of eigenvalues of $S_n$ also exhibit similar asymptotic behaviour. With a motivation  from this problem, stability components are defined as follows.
\begin{definition}
If $\b_1(n), \b_2(n), \ldots , \b_d(n) $ are eigenvalues of $S_n$ such that $|\b_1(n)|\geq |\b_2(n)|\geq \cdots \geq |\b_d(n)| >0$, then the {\bf stability exponents} ${\delta_1,\delta_2 \ldots \delta_d}$ associated with $\mu$ are defined as
$$
\delta_p := \lim_{n \to \infty} \frac{1}{n} \log |\b_p(n)|
$$
for all $1 \leq p \leq d$, in the cases where the limits exist.
\end{definition}

 Before we proceed to try to  answer  this question (which is positive at least in a few cases when the Lyapunov exponents are distinct), it is required to know the conditions on measure $\mu$ ensuring the distinctness of Lyapunov exponents and also to study the asymptotic behaviour of right and left singular vectors of $S_n$.

\begin{definition} 
Let $\mu$ (resp. $\mu^*$) be the probability measure  of the invertible random matrix $Y_1$ (resp. $Y_1^*$) on $ GL_d(\R)$. Then $\bm {T_{\mu}}$ is defined to be the smallest closed semigroup in $ GL_d(\R)$ which contains the support of $\mu$ in $ GL_d(\R)$. (Support of $\mu$ in $ GL_d(\R)$ is the set of all points  in $ GL_d(\R)$ whose every open-neighbourhood has positive measure).
\end{definition}
 The following definitions are used to describe the sets $T_{\mu}$ for which the Lyapunov exponents are distinct.  
\begin{definition}
Given a subset $T$ of $ GL_d(\R)$, we define the index of $T$ as the least integer $r$ such that there exists a sequence $\{M_n, n\geq 0\}$ in $T$ for which $\|M_n\|^{-1}M_n$ converges to a rank $r$ matrix. We say that $T$ is \textbf {contracting} when its index is one. $T$ is said to be \textbf {$p$- contracting} if $\{\wedge^p M; M \in T\}$ is contracting.
\end{definition}
\begin{definition}Irreducibility,
\begin{itemize}
\item $T$ is \textbf {irreducible}, if there is no proper linear subspace $V$ of $\R^d$ such that $M(V)=V$ for any $M$ in $T$.
\item $T$ is \textbf {strongly irreducible}  if there does not exist a finite family of proper linear subspaces of $\R^d$, $V_1,V_2\ldots V_k$ such that $M(V_1\cup V_2\cup \ldots V_k)=(V_1\cup V_2\cup \ldots V_k)$ for any $M$ in $T$.
\item $T$ is \textbf {$p$-strongly irreducible} if $\{\wedge^p M; M \in T\}$ is strongly irreducible.
\end{itemize}
\end{definition}
Below is a very useful fact (see Lemma 3.3, Chapter III \cite{prm}).
\begin{fact}\label{star}
The index of $T_{\mu}$ is equal to the index of $T_{{\mu}^*}$. $T_{{\mu}^*}$ is strongly irreducible if $T_{\mu}$ is.
\end{fact}

Below are some more facts regarding the distinctness of Lyapunov exponents, which we would use   to study the asymptotic behaviour of right and left singular vectors of $S_n$. It is natural to assume that $T_\mu$ is irreducible (see exercise 5.3, Chapter I, \cite{prm}). The following fact gives necessary and sufficient conditions for the distinctness of the first two Lyapunov exponents. (See  Theorem 6.1,  Chapter III, \cite{prm} for the following fact).  
\begin{fact}\label{fact3}
Let $\mu$ be a probability measure on $GL_d(\R)$ such that $\int \log ^+ \|M\| d{\mu}(M)$ is finite. We suppose $T_\mu$ is irreducible. Then $\gamma_1 > \gamma_2$ if and only if $T_\mu$ is strongly irreducible and contracting. 
\end{fact}

\begin{remark}\label{remark2}
If $\mu$ is a probability measure on $GL_d(\R)$ such that $\int \log ^+ \|M\| d{\mu}(M)$ is finite and   $T_\mu$ is $p$-strongly irreducible and $p$-contracting for some $p \in \{1,\ldots, d-1\}$, then by applying the above fact to the  $p$-th exterior power $\wedge^p M$ of  random matrix $M$ with distribution $\mu$, it follows from Remark \ref{remark1} that  $\gamma_1+\cdots + \gamma_{p-1} + \gamma_p > \gamma_1+\cdots + \gamma_{p-1} + \gamma_{p+1} $. Therefore $\gamma_p > \gamma_{p+1}$.
\end{remark}

\begin{fact}\label{fact2}
Let $Y_1,Y_2,\ldots$ be $i.i.d$ random elements of $GL_d(\R)$ with common distribution $\mu$ and $S_n=Y_n \cdots Y_1$. Consider a polar decomposition $S_n=U_n \Sigma_n V_n^*$ with $U_n$ and $V_n$ in $O(d)$ and $\Sigma_n= \mbox{diag}(\a_1(n),\dots, \a_{d}(n))$ with $\a_1(n)\geq \a_2(n)\geq \cdots \geq \a_d(n) >0$. If $T_{\mu}$ is a strongly irreducible semigroup with index $p$, then
\begin{itemize}
\item The subspace spanned by $\{V_n e_1,\ldots, V_n e_p\}$
 converges almost surely to a  $p$-dimensional random subspace $\mathbb V$.

\item  For any vector $u$ in $\R^d$, $\P(u \hspace{2pt} \mbox{is orthogonal to}  \hspace{2pt} \mathbb V)=0$. 

\item With probability one,
$$
\lim_{n \to \infty} \frac{\a_{p+1}(n)}{\|S_n\|}=0.
$$
\end{itemize}
\end{fact}

See Theorem 3.1, Proposition 3.2, Chapter III, \cite{prm} for the above fact.  From the above facts \ref{fact1}, \ref{fact3} and \ref{fact2}, we can see that if $\int \log ^+ \|M\| d{\mu}(M)$ is finite and  $T_\mu$ is strongly irreducible and contracting, then  
\begin{itemize}
\item ${V_n e_1}$
 converges almost surely to a   random vector $\bm v$.

\item  For any vector $u$ in $\R^d$, $\P(u \hspace{2pt} \mbox{is orthogonal to}  \hspace{2pt} \bm v)=0$. 

\item With probability one,
$$
\lim_{n \to \infty} \l(\frac{\a_{2}(n)}{\a_{1}(n)}\r)^{\frac{1}{n}}=e^{\gamma_2-\gamma_1}<1.
$$
\end{itemize}

Now we shall see that $V_n e_1$ converges exponentially fast to $\bm v$. We would also see that the direction of $V_n e_1$ converges in distribution to that of $\bm v$ exponentially fast. The proof for exponential  convergence  in distribution of $V_n e_1$ follows from that of $U_n e_1$, which we would prove by showing that the directions of $U_n e_1$ and $S_n e_1$ converge to each other exponentially fast. Exponential convergence in distribution  of the direction of $S_n e_1$   is already known. We will recall few definitions from the literature, for the sake of completeness.
 
 We say that two  non-zero vectors $x,y$ in $\R^d$ have the same direction if for some $\lambda \in \R$, $x=\lambda y$. This defines an equivalence relation $\Gamma$ on $\R^d- \{0\}$.
\begin{definition}
The set of directions in $\R^d$ is the projective space  $\P(\R^d)$ defined as the quotient space $\R^d-\{0\}/\Gamma$. For $u \in \R^d- \{0\}$, $\bar{u}$ denotes its direction (i.e. its class in $\P(\R^d)$). For $M$ in $GL_d(\R)$, we set $\bm{M.\bar{u}=\bar{Mu}}$ $i.e.$ $M.\bar{u}$ denotes a unit vector along the direction of $Mu$.
\end{definition}
 If $u,v$ are two unit vectors in $\R^d$, we define 
 $$
\delta(\bar{u},\bar{v})=\sqrt{1-|\langle u,v \rangle |^2},
 $$
where $\langle, \rangle$ denotes the usual Euclidean inner product on $\R^d$. $\delta(\bar{u},\bar{v})$ defines a metric on $\P(\R^d)$. We can identify $\bar{u}$ with the unit vector in the class  $\bar{u}$, whose first non-zero coordinate is positive.
\begin{definition}
For a probability measure $\mu$ on $GL_d(\R)$, a probability measure $\nu$ on  $\P(\R^d)$ is said to be $\bm{\mu}$-$\bm{\mbox{invariant}}$ if the distribution of  $M.\bm{\bar u}$ is  also  $\nu$  when $M$ is independent  of $\bm{\bar {u}}$ and the distributions of $M$ and $ \bm {\bar u}$ are $\mu$ and $\nu$ respectively.
\end{definition}
The first Lyapunov exponent associated with $\mu$ can be computed with the help  of any  $\mu$-invariant distribution on $\P(\R^d)$. 
\begin{fact}If $T_\mu$ is strongly irreducible, then the first Lyapunov exponent $\gamma_1$ associated with $\mu$ is given by
$$
\gamma_1=\iint \log{\frac{\|Mx\|}{\|x\|}} d\mu(M)d\nu(\bar x),
$$
where $\nu$ is any $\mu$-invariant distribution on $\P(\R^d)$.
\end{fact}
See Corollary 3.4, Chapter III, \cite{prm} for the above fact. 
\begin{fact} There exists a unique $\mu$-invariant distribution on $\P(\R^d)$ if $T_\mu$ is strongly irreducible and contracting.
\end{fact}  
See Theorem 3.1, Chapter III, \cite{prm} for the above fact. The following fact about  convergence of the columns of $S_n$ and the unique $\mu$-invariant distribution is the key tool used in this paper. (See Proposition 2.3, Chapter V and Theorem 2.1, Chapter VI, \cite{prm}). Recall that ${\ell(M)}= \max(\log^+\|M\|, \log^+\|M^{-1}\|).$
\begin{fact}\label{factkey}
Let $Y_1,Y_2,\ldots$ be $i.i.d$ random elements of $GL_d(\R)$ with common distribution $\mu$  such that $T_{\mu}$ is strongly irreducible and contracting. Set $S_n=Y_n\cdots Y_1$ for all $n \geq 1$. If for some $\tau >0$, $\E e^{\tau \ell(Y_1)}$ is finite, then
\begin{itemize}
\item there exists $\alpha_0 >0$  such that for each $\alpha$ in $(0,\alpha_0]$,
$$
A_{\mu, \alpha}:= \lim_{n \to \infty} {\l[ \sup_{\bar{u} \neq \bar{v}} \E\l[\l(\frac{\delta (S_n.\bar{u}, S_n.\bar{v})}{\delta (\bar{u}, \bar{v})} \r)^{\alpha} \r] \r]}^{\frac{1}{n}}< 1,
$$
\item   
there exists $\beta >0$ such that the unique $\mu$-invariant distribution $\nu$ on $\P(\R^d)$ satisfies 
$$
B_{\mu}:=\sup_{\|y\|=1} \int \l[ \frac{\|x\|}{|\langle x,y \rangle|}\r]^{\beta} d\nu(\bar{x}) < \infty.
$$
\end{itemize}
\end{fact}

\section{Lemmas and Theorems}\label{section2}
 
Now we state and prove the main lemma of this paper.
\begin{lemma}\label{lemma}
Let $Y_1,Y_2,\ldots$ be $i.i.d$ random elements of $GL_d(\R)$ with common distribution $\mu$  such that $T_{\mu}$ is strongly irreducible and contracting and for some $\tau >0$, $\E e^{\tau \ell(Y_1)}$ is finite. Set $S_n=Y_n\cdots Y_1$ and  consider a Singular value   decomposition $S_n=U_n \Sigma_n V_n^*$ with $U_n$ and $V_n$ in $O(d)$, for all $n \geq 1$. Then for any $r>0$
\begin{enumerate}
\item [(i)] for any sequence $\{\bm u_n, n \geq 1\}$ of random unit vectors in $\R^d$ such that $S_n{e_1}$ and $\bm u_n$ are independent  for each $n$, with probability one we have  that
$$
\lim_{n \to \infty} \l|\frac{\langle S_n{e_1}, \bm u_n \rangle}{\|S_n{e_1}\|} \r|^{\frac{1}{n^r}}=
\lim_{n \to \infty} |\langle U_n{e_1}, \bm u_n \rangle|^{\frac{1}{n^r}}= 1
$$

\item [(ii)]  for any sequence $\{\bm v_n, n \geq 1\}$ of random unit vectors in $\R^d$ such that $S^*_n{e_1}$ and $\bm v_n$ are independent  for each $n$, with probability one we have  that 
$$
\lim_{n \to \infty} \l|\frac{\langle S^*_n{e_1}, \bm v_n \rangle}{\|S^*_n{e_1}\|} \r|^{\frac{1}{n^r}}=
\lim_{n \to \infty} |\langle V_n{e_1}, \bm v_n \rangle|^{\frac{1}{n^r}}= 1.
$$
\end{enumerate}
\end{lemma}
\begin{proof}
For $0\leq x \leq 1$, $x^{\frac{1}{a}}\leq x^{\frac{1}{b}} \leq 1$ if $0<a<b$. So, it is enough to prove the lemma  for  $0<r<1$.

{\bf Proof of part(i):}

Let $\bm{w}$ be a random unit vector, independent of $\{Y_n, n \geq 1\}$ and $\{\bm u_n, n \geq 1\}$, such that the  distribution of $\bm{\bar{w}}$ is  the unique $\mu$-invariant distribution $\nu$ on $\P(\R^d)$. By definition, the distribution of $S_n.\bm {\bar{w}}$ is  $\nu$ for each $n$. Since $\delta(\bar{u},\bar{v})$ is a metric on $\P(\R^d)$, we have  for any $0\leq a_n < b_n\leq 1$ and a fixed unit vector $u$,
\begin{align*}
\P(\delta(S_n. \bar{e_1}, \bar{u}) \geq b_n) &\leq \P(\delta(S_n. \bar{e_1}, S_n. \bm{\bar{w}}) \geq b_n-a_n)  + \P(\delta(S_n. \bm{\bar{w}}, \bar{u}) \geq a_n)\\
&\leq \frac{1}{b_n-a_n} \E(\delta(S_n. \bar{e_1}, S_n. \bm{\bar{w}}))+ \P(\delta(S_n. \bm{\bar{w}}, \bar{u})\geq a_n)\\
&= \frac{1}{b_n-a_n} \int \E(\delta(S_n. \bar{e_1}, S_n. {\bar{w}})) d\nu(\bar{w})+ \P(\delta(\bm{\bar{w}}, \bar{u})\geq a_n)\\ 
&\mbox{choose $\alpha <1$ and since $\delta(\bar{u},\bar{v})=\sqrt{1-|\langle u,v \rangle |^2 }\leq 1$}\\
&\leq \frac{1}{b_n-a_n} \int \E\l[\l(\frac{\delta (S_n.\bar{e_1}, S_n.\bar{w})}{\delta (\bar{e_1}, \bar{w})} \r)^{\alpha}\r]d\nu(\bar{w})+ \P(|\langle\bm{{w}}, u \rangle | \leq \sqrt{1-a_n^2} )\\
&\leq  \frac{1}{b_n - a_n} \sup_{\bar{u} \neq \bar{v}} \E\l[\l(\frac{\delta (S_n.\bar{u}, S_n.\bar{v})}{\delta (\bar{u}, \bar{v})} \r)^{\alpha} \r]+ {(1-a_n^2)^\frac{\beta}{2}}\sup_{\|y\|=1} \int \l[ \frac{1}{|\langle w,y \rangle|}\r]^{\beta} d\nu(\bar{w}).
\end{align*}
 Notice that the upper bounds we got for $\P(\delta(S_n. \bar{e_1}, \bar{u}) \geq b_n)$ do not depend on the vector $u$ for every $n$.  
It follows  from the previous fact \ref{factkey} that, for $\alpha < \min(1, \alpha_0)$
$$
\limsup_{n \to \infty} \l[\frac{1}{b_n-a_n} \sup_{\bar{u} \neq \bar{v}} \E\l[\l(\frac{\delta (S_n.\bar{u}, S_n.\bar{v})}{\delta (\bar{u}, \bar{v})} \r)^{\alpha} \r]\r]^{\frac{1}{n}}= \limsup_{n \to \infty} \l[\frac{1}{b_n-a_n}\r]^{\frac{1}{n}} A_{\mu, \alpha}
$$ 
Therefore if 
$$\limsup_{n \to \infty} \l[\frac{1}{b_n-a_n}\r]^{\frac{1}{n}} A_{\mu, \alpha}< 1 \hspace{5pt}\mbox{and} \hspace{5pt} \sum_n  {(1-a_n^2)^\frac{\beta}{2}} < \infty,$$
 then 
 $$\sum_{n=1}^{\infty} \sup_{\|u\|=1}\P\l(\frac{|\langle S_n{e_1}, u \rangle|}{\|S_n{e_1}\|} \leq \sqrt{1-b_n^2}\r)=\sum_{n=1}^{\infty} \sup_{\|u\|=1} \P(\delta(S_n. \bar{e_1}, \bar{u}) \geq b_n) < \infty.$$ 
Given $\epsilon > 0$, we can choose $b_n=\sqrt{1-e^{-2n^r\epsilon}}$ and $a_n=\sqrt{1-e^{-2n^r\epsilon_1}}$ for some $0< \epsilon_1< \epsilon$ and $0<r<1$. Then clearly
$$\limsup_{n \to \infty} \l[\frac{1}{b_n-a_n}\r]^{\frac{1}{n}} A_{\mu, \alpha}=  A_{\mu, \alpha}<1 \hspace{5pt}\mbox{and} \hspace{5pt} \sum_n  {(1-a_n^2)^\frac{\beta}{2}} < \infty.$$
Therefore for every $\epsilon > 0$ and $0<r<1$, 
$$\sum_{n=1}^{\infty} \sup_{\|u\|=1} \P\l(\frac{|\langle S_n{e_1}, u \rangle|}{\|S_n{e_1}\|} \leq e^{-n^r\epsilon}\r)< \infty.
$$
Now for a sequence $\{\bm u_n, n \geq 1\}$ of random unit vectors in $\R^d$ such that $S_n{e_1}$ and $\bm u_n$ are independent  for each $n$,  we have  that
$$\sum_{n=1}^{\infty}  \P\l(\frac{|\langle S_n{e_1}, \bm u_n \rangle|}{\|S_n{e_1}\|} \leq e^{-n^r\epsilon}\r)< \infty.
$$
By second Borel-Cantelli lemma, we conclude that with probability one for $0<r<1$, 
\begin{eqnarray}\label{limit}
\lim_{n \to \infty} \l|\frac{\langle S_n{e_1}, \bm u_n \rangle}{\|S_n{e_1}\|} \r|^{\frac{1}{n^r}}=1.
\end{eqnarray}

Recall  the singular value decomposition   $S_n=U_n \Sigma_n V_n^*$ with $U_n$ and $V_n$ in $O(d)$, for all $n \geq 1$. Let 
$\Sigma_n= \mbox{diag}(\a_1(n),\dots, \a_{d}(n))$ with $\a_1(n)\geq \a_2(n)\geq \cdots \geq \a_d(n) >0$. Now see that
\begin{align*}
\l\|\frac{S_n e_1}{\|S_n e_1\|}- U_n e_1\r\|^2 &\leq 2\l[1-\l|\frac{\langle S_n{e_1}, U_n e_1 \rangle}{\|S_n{e_1}\|}\r|^{2}\r]\\
&= 2\l[1-\frac{\a^2_1(n)|{\langle e_1,V_n e_1\rangle}|^2}{\sum_{i=1}^d \a^2_i(n){|\langle e_i,V_n e_1\rangle}|^2}\r] \\
&=2\frac{\sum_{i=2}^d \a^2_i(n){|\langle e_i,V_n e_1\rangle}|^2}{\sum_{i=1}^d \a^2_i(n){|\langle e_i,V_n e_1\rangle}|^2}\\
&\leq \frac{2}{{|\langle e_1,V_n e_1\rangle}|^2}\frac{\a^2_2(n)} {\a^2_1(n)}
\end{align*} 

From the facts  \ref{fact1}, \ref{fact2} and \ref{fact3}, we know that 
$V_n e_1$ converges almost surely to a random vector $\bm v$ and  with probability  one
$$
\lim_{n \to \infty} |\langle e_1,V_n e_1\rangle| = |\langle e_1,\bm  v\rangle| \neq 0 \hspace{5pt} \mbox{and} \hspace{5pt} \lim_{n \to \infty} \l(\frac{\a_{2}(n)}{\a_{1}(n)}\r)^{\frac{1}{n^r}}=0<1.
$$
Therefore with probability one,
\begin{eqnarray}\label{limit2}
\limsup_{n \to \infty} \l|\frac{\langle S_n{e_1}, \bm u_n \rangle}{\|S_n{e_1}\|}- \langle U_n{e_1}, \bm u_n \rangle \r|^{\frac{1}{n^r}} \leq \limsup_{n \to \infty} \l\|\frac{S_n e_1}{\|S_n e_1\|}- U_n e_1\r\|^{\frac{1}{n^r}}=0  < 1
\end{eqnarray}
If $\lim_{n \to \infty} |a_n|^{\frac{1}{n^r}}=1$ and $\limsup_{n \to \infty} |b_n|^{\frac{1}{n^r}}<1$, then $\lim_{n \to \infty} |a_n+b_n|^{\frac{1}{n^r}}=1$.  It follows from here, using equations  \eqref{limit} and \eqref{limit2}, that 
$$\lim_{n \to \infty} |\langle U_n{e_1}, \bm u_n \rangle|^{\frac{1}{n^r}}= 1
$$

\noindent {\bf Proof of part(ii):}

Let $\mu^*$ be the distribution of $Y^*_1$.  Then by fact \ref{star}, $T_{\mu ^*}$ is also strongly irreducible and contracting. Since $\ell(Y^*_1)=\ell(Y_1)$, $\E e^{\tau l(Y^*_1)}$ is also finite. Therefore $\mu^*$ satisfies the conditions in the hypothesis of this lemma. Consider the matrices $Y^*_1,Y^*_2, \ldots $ and set $M_n=Y^*_n \cdots  Y^*_1$ for each $n$. Then by the proof of part $(i)$ of this lemma, we have that for every $\epsilon >0$, 
$$\sum_{n=1}^{\infty} \sup_{\|v\|=1} \P\l(\frac{|\langle M_n{e_1},  v \rangle|}{\|M_n{e_1}\|} \leq e^{-n^r\epsilon}\r)< \infty.
$$
The distribution of $M_n$ is the same as that of $S^*_n$, so
$$\sum_{n=1}^{\infty} \sup_{\|v\|=1} \P\l(\frac{|\langle S^*_n{e_1},   v \rangle|}{\|S^*_n{e_1}\|} \leq e^{-n^r\epsilon}\r)< \infty.
$$
For any sequence $\{\bm v_n, n \geq 1\}$ of random unit vectors in $\R^d$ such that $S^*_n{e_1}$ and $\bm v_n$ are independent  for each $n$, we have that  
$$\sum_{n=1}^{\infty}  \P\l(\frac{|\langle S^*_n{e_1}, \bm v_n \rangle|}{\|S^*_n{e_1}\|} \leq e^{-n^r\epsilon}\r)< \infty.
$$
By second Borel-Cantelli lemma, we conclude that with probability one, 
\begin{eqnarray}\label{limit3}
\lim_{n \to \infty} \l|\frac{\langle S^*_n{e_1}, \bm v_n \rangle}{\|S^*_n{e_1}\|} \r|^{\frac{1}{n^r}}=1.
\end{eqnarray}
Now by a computation done earlier in the proof, we have    
$$\l\|\frac{S^*_n e_1}{\|S^*_n e_1\|}- V_n e_1\r\|^2 \leq \frac{2}{{|\langle e_1,U_n e_1\rangle}|^2}\frac{\a^2_2(n)} {\a^2_1(n)}.$$ By part $(i)$ of this lemma, we have with probability one 
$$
\lim_{n \to \infty} |\langle U_n{e_1}, e_1 \rangle|^{\frac{1}{n^r}}= 1,
$$
which implies that with probability one,
\begin{eqnarray}\label{limit5}
\limsup_{n \to \infty} \l|\frac{\langle S^*_n{e_1}, \bm v_n \rangle}{\|S^*_n{e_1}\|}- \langle V_n{e_1}, \bm v_n \rangle \r|^{\frac{1}{n^r}} \leq \limsup_{n \to \infty} \l\|\frac{S^*_n e_1}{\|S^*_n e_1\|}- V_n e_1\r\|^{\frac{1}{n^r}}=0 < 1.
\end{eqnarray} 
From the equations \eqref{limit3} and \eqref{limit5}, we can conclude that with probability one
$$\lim_{n \to \infty} |\langle V_n{e_1}, \bm v_n \rangle|^{\frac{1}{n^r}}= 1
$$

\end{proof} 
Now we proceed to the main theorem of this paper which states  conditions on the  measure $\mu$ for the first Lyapunov  and stability exponents to match.
\begin{theorem}
Let $Y_1,Y_2,\ldots$ be $i.i.d$ random elements of $GL_d(\R)$ with common distribution $\mu$  such that $T_{\mu}$ is strongly irreducible and contracting. Set $S_n=Y_n\cdots Y_1$ and  let $\b_1(n)$ be an eigenvalue of $S_n$ with maximum absolute value and $\a_1(n)$ be the maximum singular value of $S_n$, for all $n \geq 1$. If for some $\tau >0$, $\E e^{\tau \ell(Y_1)}$ is finite, then for any $r>0$ with probability one
$$
\lim_{n \to \infty} \frac{1}{n^r}\log \l(\frac{|\b_1(n)|}{\a_1(n)}\r)= 0.
$$
\end{theorem}

\begin{proof}

Since $|\Tr(S_n)| \leq d |\b_1(n)| \leq d\a_1(n)$, it is enough to show that, with probability one
$$
\lim_{n \to \infty} \frac{1}{n^r}\log \frac{|\Tr(S_n)|}{\a_1(n)}=0.
$$
Also, it is enough to show for $0<r<1$.

For $n>k\geq  0$, set $S_{n,k}= Y_n\cdots Y_{k+1}$. See that $S_{n,0}=S_n$ and $S_{n,0}=S_{n,k}S_{k,0}$. Consider a singular value decomposition   $S_{n,k}=U_{n,k} \Sigma_{n,k} V_{n,k}^*$ with $U_{n,k}$ and $V_{n,k}$ in $O(d)$ and 
$\Sigma_{n,k}= \mbox{diag}(\a_1(n,k),\dots, \a_{d}(n,k))$ with $\a_1(n,k)\geq \a_2(n,k)\geq \cdots \geq \a_d(n,k) >0$, for all $n>k \geq  0$. 
Observe that 
\begin{align*}
\Tr(S_n) &= \Tr(S_{n,k}S_{k,0})=\Tr(U_{n,k} \Sigma_{n,k} V_{n,k}^*U_{k,0} \Sigma_{k,0} V_{k,0}^*)\\
&=\Tr( \Sigma_{n,k} V_{n,k}^*U_{k,0} \Sigma_{k,0} V_{k,0}^*U_{n,k})\\
&= \sum_{i=1}^d \a_i(n,k) \langle e_i, V_{n,k}^*U_{k,0} \Sigma_{k,0} V_{k,0}^*U_{n,k}e_i\rangle\\
&= \sum_{i=1}^d \sum_{j=1}^d \a_i(n,k) \langle e_i, V_{n,k}^*U_{k,0}e_j \rangle \a_j(k,0) \langle e_j, V_{k,0}^*U_{n,k}e_i\rangle\\
&= \sum_{i=1}^d \sum_{j=1}^d \a_i(n,k)\a_j(k,0) \langle V_{n,k}e_i, U_{k,0}e_j \rangle  \langle V_{k,0}e_j, U_{n,k}e_i\rangle.
\end{align*}

By Fact \ref{fact1}, we have  with probability one, for all $1 \leq p \leq d$,
\begin{equation}\label{eqn}
\lim_{n \to \infty} \frac{1}{n}\log \a_p(n,0)= \gamma_p.
\end{equation}
Since $\a_d^2(k,0)S_{n,k}S^*_{n,k} \leq S_nS_n^* \leq \a_1^2(k,0)S_{n,k}S^*_{n,k}$,
 we can see by Min-max theorem that for all $1 \leq p \leq d$,
\begin{equation}\label{eq100}
\a_p(n,k) \a_d(k,0) \leq \a_p(n,0) \leq \a_p(n,k) \a_1(k,0).\end{equation}
Since $\E e^{\tau \ell(Y_1)}$ is finite, $\E(\log^{+}\|Y_1^{-1}\|)$ is also finite  which implies that  all the Lyapunov exponents $\gamma_1,\gamma_2, \ldots , \gamma_d$ associated with $\mu$ are finite.
Let $k \to \infty$ and $n \to \infty$ such that $\frac{k}{n^r} \to 1$ for $0<r<1$, then  by  \eqref{eqn} and \eqref{eq100} we have with probability one, for all $1 \leq p \leq d$,
\begin{equation}\label{eq1}
\lim_{{n \to \infty} } \frac{1}{n}\log \a_p(n,k)= \gamma_p.
\end{equation}

  Since  $T_{\mu}$ is strongly irreducible and contracting, by Fact \ref{fact3}, $\gamma_1> \gamma_2$. This implies that for $i \neq 1$ or $j \neq 1$ with probability one as  $\frac{k}{n^r} \to 1$ for $0<r<1$
\begin{equation}\label{eq120}
\lim_{n \to \infty} \l|\frac{\a_i(n,k)\a_j(k,0)}{\a_1(n,k)\a_1(k,0)}\r|^{\frac{1}{n^r}} \leq \frac{\gamma_2}{\gamma_1}<1.
\end{equation}
Since $V_{n,k}e_1$, $U_{n,k}e_1$  are independent of  $U_{k,0}e_1$,$  V_{k,0}e_1$ for every $n>k>0$, by previous lemma \ref{lemma} as $k \to \infty$ and $n \to \infty$ such that $\frac{k}{n^r} \to 1$ for $0<r<1$ we get with probability one that 
\begin{equation}\label{121}
\lim_{n \to \infty} \l| \langle V_{n,k}e_1, U_{k,0}e_1 \rangle  \langle V_{k,0}e_1, U_{n,k}e_1\rangle\r|^{\frac{1}{n^r}}=1
\end{equation}
  
If $\lim_{n \to \infty} |a_n|^{\frac{1}{n^r}}=1$ and $\limsup_{n \to \infty} |b_n|^{\frac{1}{n^r}}<1$, then $\lim_{n \to \infty} |a_n+b_n|^{\frac{1}{n^r}}=1$.
So, using \eqref{eq120} and \eqref{121} in $\Tr(S_n)$ we get that with probability one
$$
\lim_{n \to \infty} \l| \frac{|\Tr(S_n)|}{\a_1(n,k)\a_1(k,0)}\r|^{\frac{1}{n^r}}=1.
$$
Since $|\Tr(S_n)| \leq d |\b_1(n)| \leq d\a_1(n) \leq d\a_1(n,k) \a_1(k,0) $, we get that with probability one
$$
\lim_{n \to \infty} \l| \frac{\b_1(n)}{\a_1(n)}\r|^{\frac{1}{n^r}}=1.
$$
This concludes the proof.
\end{proof}

\begin{remark}
Under the hypothesis of the above theorem, it is known that (see Theorem 5.1 Chapter V \cite{prm}) $\sqrt{n}(\frac{1}{n} \log \a_1(n) - \gamma_1)$ converges in distribution to $N(0,a^2)$ for some $a>0$. From the above theorem, by taking $r= \frac{1}{2}$, we get that $\sqrt{n}(\frac{1}{n} \log |\b_1(n)| - \gamma_1)$ also converges in distribution to $N(0,a^2)$.
\end{remark}

With further conditions on $T_{\mu}$, we can prove the result for all the eigenvalues of $S_n$. 
\begin{theorem}\label{theorem2}
Let $Y_1,Y_2,\ldots$ be $i.i.d$ random elements of $GL_d(\R)$ with common distribution $\mu$  such that $T_{\mu}$ is $p$-strongly irreducible and $p$-contracting for all $1 \leq p \leq d$. Set $S_n=Y_n\cdots Y_1$ and  let $\b_1(n), \b_2(n), \ldots ,\b_d(n)$ be  the eigenvalues of $S_n$ such that $|\b_1(n)| \geq |\b_2(n)| \geq \cdots \geq |\b_d(n)|>0$ and $\a_1(n), \a_2(n), \ldots ,\a_d(n)$ be  the singular values of $S_n$ such that $|\a_1(n)| \geq |\a_2(n)| \geq \cdots \geq |\a_d(n)|>0$, for all $n \geq 1$. If for some $\tau >0$, $\E e^{\tau \ell(Y_1)}$ is finite, then for any  $r>0$ with probability one 
$$
\lim_{n \to \infty} \frac{1}{n^r}\log \l(\frac{|\b_p(n)|}{\a_p(n)}\r)= 0,
$$
for all $1 \leq p \leq d$.
\end{theorem}

\begin{proof}
The distinctness of the Lyapunov exponents follows from Remark \ref{remark2}.
 Since $\frac{1}{p} l(\wedge^p M) \leq \ell(M)$ and  $T_{\mu}$ is $p$-strongly irreducible and $p$-contracting, for any $1 \leq p \leq d$, the sequence $\{\wedge^p Y_n, n \geq 1\}$ of random matrices satisfies the conditions of the  previous Theorem. $\b_1(n) \b_2(n)\cdots \b_p(n)$ is an eigenvalue of $\wedge^p Y_1 \wedge^p Y_2 \cdots \wedge^p Y_n$ with maximum absolute value. Therefore, with probability one, for all $1 \leq p \leq d$
 $$
 \lim_{n \to \infty} \frac{1}{n^r}\log \l(\frac{|\b_1(n) \cdots \b_p(n)|}{\a_1(n) \cdots \a_p(n)}\r)= 0,
 $$
 which proves the theorem.
\end{proof}
 
 The following facts give us some  criteria which can be used to check if $T_{\mu}$ is $p$-strongly irreducible and $p$-contracting, for any $1 \leq p \leq d$. 
 
\begin{fact}
$T_{\mu}$ is $p$-strongly irreducible and $p$-contracting, for any $1 \leq p \leq d$, if there exists a matrix $M$ such that 
\begin{itemize}
\item $\|M\|^{-1}M$ is not orthogonal,
\item for any $K \in O(d)$, $KMK^{-1}$ is in $T_{\mu}$.
\end{itemize}
\end{fact}
See proposition 2.5, Chapter IV, \cite{prm} for the above fact. In the case of isotropic random matrices, the above fact can be used  to verify the conditions of the theorem. 
 
 \begin{fact}\label{magic}
 If $\{|\det{M}|^{-\frac{1}{d}}M; M \in T_{\mu}\}$ contains an open set of $\{M \in GL_d(\R); |\det M|=1 \}$, then $T_{\mu}$ is $p$-strongly irreducible and $p$-contracting, for any $1 \leq p \leq d$.
 \end{fact}
See proposition 2.3, Remart 2.4, Chapter IV, \cite{prm} for the above fact.
Matrices with $i.i.d$ elements are the most   basic examples of random matrix theory.  The following corollary gives conditions on  the distribution of matrix elements for the main theorem to hold.

\begin{corollary}
Let $\xi$ be a real valued continuous  random variable whose support contains an open set and there exists $\tau>0$ such that 
$$
\E |\xi|^{\tau} < \infty \hspace{5pt} \mbox{and} \hspace{5pt} \sup_{a \in \R} \E \l|{\xi-a}\r|^{-\tau} < \infty.$$ Let $X_1,X_2, \ldots $ be a sequence of $d \times d $ $i.i.d$ random matrices whose elements are $i.i.d$ random variables distributed like $\xi$ and  $\b_1(n), \b_2(n), \ldots , \b_d(n)$ be  the eigenvalues, $\a_1(n), \a_2(n), \ldots , \a_d(n)$ be  the singular values of $X_nX_{n-1}\cdots X_1$ in the decreasing order of their absolute values for every $n$. Then for any $r>0$ with probability one for all $1 \leq p \leq d$,
$$
\lim_{n \to \infty} \frac{1}{n^r}\log \l(\frac{|\b_p(n)|}{\a_p(n)}\r)= 0.
$$
\end{corollary} 
 \begin{proof}
 Since $\xi$ is  continuous random variable, $X_1$ is invertible with probability one. For each $n\geq 1$, we can get an invertible random matrix $Y_n$  such that $X_n$ and $Y_n$ are equal with probability one, by defining $Y_n$ to be the identity matrix  on the set where $X_n$ is not invertible.
Let $\mu$ be the probability measure of $Y_1$ on $GL_d(\R)$ and $B$ be an open interval contained in the support of $\xi$. By the definition of $X_1$ and $Y_1$, we can see that the invertible matrices, whose all  elements are from  $B$, belong to the support of $\mu$. we can choose two numbers $a$ and $b$ from $B$ such that $M_0= (b-a)I+aJ$, where $I$ is the identity matrix  and $J$ is the matrix with all its elements equal to one, is invertible. Clearly $M_0$ is in the support of $\mu$. We can choose $\epsilon>0$ small enough such that every matrix in the $\epsilon$-neighbourhood of $M_0$ is invertible and has all its elements from $B$. That is, $\epsilon$-neighbourhood of $M_0$ is in the support of $\mu$. For every matrix $N$ in $$\mathbb A=\{N \in GL_d(\R); \|N-|\det{M_0}|^{-\frac{1}{d}}M_0\|<|\det{M_0}|^{-\frac{1}{d}}\epsilon, |\det{N}|=1\},$$ we have that $|\det{M_0}|^{\frac{1}{d}}N$ is in the $\epsilon$-neighbourhood of $M_0$, therefore in the support of $\mu$. Therefore
$$
\mathbb A \subseteq \{|\det{M}|^{-\frac{1}{d}}M; M \in T_{\mu}\}
$$ and $\mathbb A$ is an open set of $\{N \in GL_d(\R); |\det N|=1 \}$. By Fact \ref{magic}, $T_{\mu}$ is $p$-strongly irreducible and $p$-contracting, for any $1 \leq p \leq d$. 

Now in order to apply  Theorem \ref{theorem2} to the sequence of random matrices $\{Y_n, n\geq 1\}$, we need to show that for some $\tau> 0$, $\E e^{\tau \ell(Y_1)}$ is finite. Let $X$ be a random matrix with $i.i.d$ elements distributed like $\xi$.  Observe that
\begin{align*}
\E e^{\tau \ell(Y_1)} &= \E e^{\tau \ell(X)}= \E e^{\max (\log^+\|X\|, \log^+\|X^{-1}\|)}\\
&\leq \E \|X\|^{\tau} + \E \|X^{-1}\|^{\tau}\\
&\leq \E \l[\sum_{i,j=1}^d |X_{i,j}|\r]^{\tau} + \E \l[\sum_{i,j=1}^d |X^{-1}_{i,j}|\r]^{\tau}\\
& \hspace{5pt} (\mbox {where $X^{-1}_{i,j}$ is $i,j$-th element of $X^{-1}$})\\
&\leq d^{2\tau}\E \l[\max_{i,j} |X_{i,j}|\r]^{\tau} + d^{2\tau}\E \l[\max_{i,j} |X^{-1}_{i,j}|\r]^{\tau}\\
&= d^{2\tau}\E \max_{i,j} |X_{i,j}|^{\tau} + d^{2\tau}\E \max_{i,j} |X^{-1}_{i,j}|^{\tau}\\
&\leq  d^{2\tau }\E  \l[\sum_{i,j=1}^d |X_{i,j}|^{\tau}\r] + d^{2\tau }\E \l[\sum_{i,j=1}^d |X^{-1}_{i,j}|^{\tau}\r]\\
&= d^{2\tau +2}\E  |\xi|^{\tau} + d^{2\tau +2 }\E  |X^{-1}_{1,1}|^{\tau}.
\end{align*} Therefore we have,
$$
\frac{1}{2}[\E  |\xi|^{\tau} + \E  |X^{-1}_{1,1}|^{\tau}] \leq \E e^{\tau \ell(Y_1)} \leq d^{2\tau +2}\l[\E  |\xi|^{\tau} + \E  |X^{-1}_{1,1}|^{\tau}\r].
$$
If $X = \l(\begin{array}{cc}
\xi &   u^*\\
  v & C
\end{array}\r)$, then $X^{-1}_{1,1}= (\xi- u^*C^{-1} v)^{-1}$. Notice that $ u, v,C$ are independent of $\xi.$ Therefore by the given conditions on $\xi$, we have 
$$\frac{1}{d^{2\tau +2}}\E e^{\tau \ell(Y_1)} \leq \E  |\xi|^{\tau} +\E  |X^{-1}_{1,1}|^{\tau} \leq \E  |\xi|^{\tau}+\sup_{a \in \R} \E \l|{\xi-a}\r|^{-\tau} < \infty. $$ 
It is clear that  $\b_1(n), \b_2(n), \ldots , \b_d(n)$ are  the eigenvalues and  $\a_1(n), \a_2(n), \ldots , \a_d(n)$ are  the singular values of  $Y_nY_{n-1}\cdots Y_1$ for each $n$, with probability one. Since the distribution $\mu$ of $Y_1$ satisfies the conditions of Theorem \ref{theorem2},for any $r>0$ with probability one
$$
\lim_{n \to \infty} \frac{1}{n^r}\log \l(\frac{|\b_p(n)|}{\a_p(n)}\r)= 0.
$$
for all $1 \leq p \leq d$.
\end{proof}

\begin{remark}
If the random variable $\xi$ has bounded density, then  $\sup_{a \in \R} \E \l|{\xi-a}\r|^{-\tau} < \infty$ for any $0<\tau<1$. Let us say $f$ is the probability  density of $\xi$ ,bounded by $K$, then 
\begin{align*}
 \E \l|{\xi-a}\r|^{-\tau} &= \sup_{a \in \R} \int |x-a|^{-\tau} f(x)dx\\
&=  \int_{|x-a|<\epsilon} |x-a|^{-\tau} f(x)dx +  \int_{|x-a|\geq \epsilon} |x-a|^{-\tau} f(x)dx  \hspace{17pt} (\epsilon >0)\\
& \leq  K \int_{|x-a|<\epsilon} |x-a|^{-\tau} dx + \epsilon^{-\tau}\\
&= K \frac{2\epsilon^{1-\tau}}{1-\tau}+ \epsilon^{-\tau} < \infty.
\end{align*}
Since $\epsilon, K, \tau$ do not depend on $a$, we get that
 $$\sup_{a \in \R} \E \l|{\xi-a}\r|^{-\tau} < \infty.$$
\end{remark}

\noindent{\bf Acknowledgments:} The author is  grateful to Prof. Manjunath Krishnapur for suggesting this problem.

\bibliography{bibtex}
\bibliographystyle{amsplain}

\end{document}